\documentclass[preprint,11pt]{elsarticle}

\usepackage{amssymb}
\usepackage{graphicx}
\usepackage{epsfig}
\usepackage{psfig}
\usepackage{subfigure}% subcaptions for subfigures
\usepackage{subfigmat}% matrices of similar subfigures, aka small mulitples
\usepackage{amsthm}
\usepackage{latexsym}
\usepackage{amsmath}
\usepackage{lineno}
\newtheorem{theorem}{Theorem}[section]

\newtheorem{lem}{Lemma}[section]

\numberwithin{equation}{section}
\begin{document}

\begin{frontmatter}


% use the tnoteref command within \title for footnotes;
% use the tnotetext command for theassociated footnote;
% use the fnref command within \author or \address for footnotes;
% use the fntext command for theassociated footnote;
% use the corref command within \author for corresponding author footnotes;
% use the cortext command for theassociated footnote;
  \title{Least-squares spectral element preconditioners for fourth order elliptic problems}
%  \tnotetext[label1]{This work was carried out at the Department of Mathematics, Indian Institute of Technology Kanpur,
%  Kanpur-208016, India}
%  \author{Name\corref{cor1}\fnref{label2}}
%  \ead{email address}
% \address{Address\fnref{label3}}

%\title{Lagrange multiplier method with penalty for elliptic and parabolic interface problems}
% \tnotetext[1]{Lagrange multiplier method with penalty}
\author[label2]{Akhlaq~Husain\corref{cor1}}
\ead{akhlaq.husain@bml.edu.in}
\address[label2]{School of Engineering \& Technology, BML Munjal University, Gurgaon-122413, Haryana, India}
\author[label5]{Arbaz~Khan\corref{cor2}}
\ead{arbazkha@gmail.com,arbaz.khan@iwr.uni-heidelberg.de}
%\address[label3]{Department of Mathematics \& Statistics, Indian Institute of Technology, Kanpur 208016, UP, India.}
%\address[label4]{Mathematics Center Heidelberg (MATCH), Ruprecht-Karls-Universit$\ddot{a}$t Heidelberg, 69120 Heidelberg, Germany}
\address[label5]{Interdisziplin$\ddot{a}$res Zentrum f$\ddot{u}$r Wissenschaftliches Rechnen (IWR), Ruprecht-Karls-Universit$\ddot{a}$t Heidelberg, 69120 Heidelberg, Germany}
\cortext[cor1]{This work was carried out during the second author's stay at the LNM Institute of Information Technology
(LNMIIT), Jaipur as assistant professor}
\cortext[cor2]{Corresponding author}

\begin{abstract}
% Text of abstract
The goal of this paper is to propose preconditioners for the system of linear equations that
arises from a discretization of fourth order elliptic problems using spectral element methods.
These preconditioners are constructed using separation of variables and can be diagonalized
and hence easy to invert. For second order elliptic problems this technique has proven to be
very successful and performs better than other preconditioners. We show that these preconditioners
are spectrally equivalent to the quadratic forms by which we approximate them. Numerical result
for the biharmonic problem are presented to validate the theoretical estimates.
\end{abstract}

\begin{keyword}

Fourth order problems\sep preconditioners \sep spectral element method \sep separation of variables
\sep spectral equivalence \sep condition number.


% MSC codes here, in the form: \MSC code \sep code
% or \MSC[2008] code \sep code (2000 is the default)

\end{keyword}

\end{frontmatter}
%\linenumbers
% main text
% \section{}
% \label{}

\section{Introduction}
In this paper we investigate methods for preconditioning the system of linear equations that arises
from a discretization of fourth order elliptic problems using spectral element methods. Preconditioners
are usually constructed and analyzed with the goal of maintaining a well-conditioned system of equations
as the number of unknowns $W$ increases.

Let $\Omega\subset\mathbb{R}^d$ ($d=2,3$) be a bounded curvilinear domain having smooth boundary.
We consider the model fourth order elliptic problem with homogeneous Dirichlet boundary conditions:
\begin{align}\label{eq1.1}
{\mathcal L}u=\Delta^{2}u-\nabla.(a\nabla u)+({b}.\nabla)u+cu&=f \quad\mbox{in}\quad\Omega,\notag\\
u & = \frac{\partial u}{\partial n}=0 \quad \mbox{on} \quad \partial \Omega.
\end{align}
Here, $\partial\Omega$ is the boundary %${\mathcal L}$ is a fourth order linear elliptic operator,
of $\Omega$, $\frac{\partial u}{\partial n}$ is the outer normal derivative on $\partial \Omega$,
and $f\in L^2$. The coefficients ${b},c$ and the entries in the $2\times 2$ matrix $a$ are analytic.
Moreover, $a$ is symmetric, positive definite matrix.  Problem (\ref{eq1.1}) includes the classical
plate bending problem in the theory of elasticity. This type of boundary value problems arise in
structural mechanics, materials science and fluid flow. Such problems are also associated, for example,
with the Cahn-Hilliard model for phase-separation phenomena~\cite{KH}.

Suppose that (\ref{eq1.1}) is discretized using a spectral/finite element or finite difference method
characterized by a mesh-size $h$. This yields an equation
\begin{align}\label{eq1.2}
Au^h=f.
\end{align}
in a finite-dimensional approximation space $V_h$. Here, $A$ is usually symmetric and positive
definite. An effective method for solving (\ref{eq1.2}) consists of first preconditioning $A$ and
then using a convergent iterative method such as the preconditioned conjugate gradient method (PCGM).
% applied to the normal equations.
At every iteration, this
method requires the evaluation of the matrix vector multiplications $Ax$ and $Br$, where $B$ is
another symmetric, positive definite matrix, called \emph{preconditioner}. The number of iterations
required for the iterative method to converge depends on the condition number $\kappa$ defined by
\begin{align}\label{eq1.3}
\kappa=||BA||_{V_h}\:||(BA)^{-1}||_{V_h}=\frac{\lambda_{\mbox{max}}(BA)}{\lambda_{\mbox{min}}(BA)},
\end{align}
where $||\cdot||_{V_h}$ denotes the norm in $V_h$ and $\lambda_{\mbox{max}}, \lambda_{\mbox{min}}$ are the
largest, least eigenvalues respectively. It is well known that the number of iterations, required to
achieve a given tolerance in the energy norm is proportional to $\sqrt{\kappa}$. The condition number
and hence the computational cost increases rapidly as $h\rightarrow 0$ unless a suitable preconditioner
is employed. Various preconditioners were developed and analyzed in~\cite{BAN,BRAM,DBR}
(as well as references cited therein) and shown to be extremely effective.

The preconditioners described in this paper are obtained in~\cite{KAH} for solving fourth order elliptic
problems which are defined using a quadratic form which measures the $H^4$ norm of the spectral element
function representation of the solution.
These preconditioners are obtained in the same way as in~\cite{DKU,SKT2} by computing the residuals in
the normal equations, but with homogeneous boundary data and the homogenous form of the partial differential
equation. Hence, the algorithm is quite simple and easy to implement. We show that there exists a new
diagonal preconditioner using separation of variables technique.

It is shown in~\cite{KAH} that the condition number of the preconditioned system grows like $O((\ln W)^4)$,
where $W$ denotes the polynomial degree and since we have mapped all elements in the domain $\Omega$ onto
the master square $S=(-1,1)^2$, therefore, we present the numerical results for a single reference element
only. Unless otherwise stated, all the generic constants of approximation appearing in this paper are
independent of $N$ and depend on the shape regularity of $\Omega$. Here, $N$ denotes the number of elements
in $\Omega$. For computational simplicity, we assume that the degrees of the approximating polynomials are
uniform in both directions. However, we can allow for non-uniform distribution of polynomial degree in each
direction with a reduced number of degrees of freedom. This is usually done in the anisotropic case (e.g.
in presence of singularities) where one uses lower order polynomials near singularities and higher order
polynomials away from singularities to increase the effectiveness of the preconditioner.

The contents of this paper are now provided. In Section $2$ preconditioners for elliptic problems are
examined and numerical results are presented. In Section $3$ we describe solution techniques for solving
the system of linear equations arising from the spectral element discretization. Concluding remarks are
provided in Section $4$.

\section{Preconditioners}
Our construction of preconditioners is similar to that for second order elliptic and parabolic problems
(see~\cite{DBR,DKU}). We construct a preconditioner $\mathcal{B}(u)$ on each of the elements in $\Omega$.
We shall prove (as in~\cite{DBR}) that there is another quadratic form $\mathcal{C}(u)$ which is spectrally
equivalent to $\mathcal{B}(u)$ and which can be easily diagonalized using the separation of variables.
Then the matrix corresponding to the quadratic form $\mathcal{C}(u)$ will be easy to invert.

The preconditioner which needs to be examined corresponds to the quadratic form
\begin{align}\label{eq5.7}
\mathcal B({u}) =||{u}||^{2}_{H^{4}(S)}\;,
\end{align}
where ${u}={u}(\xi,\eta)$ is a polynomial of degree $W$ in $\xi$ and $\eta$ separately.
Let ${u}(\xi,\eta)$ be the spectral element function, defined on $S$, as
\begin{equation}\label{eq5.8}
u(\xi,\eta)=\sum_{i=0}^W\sum_{j=0}^W a_{i,j}L_{i}(\xi)
L_{j}(\eta).
\end{equation}
\noindent
Here, $L_{i}(\cdot)$ denotes the Legendre polynomial of degree $i$.
\newline
The quadratic form $\mathcal B({u})$ can be written as
\begin{equation}\label{eq5.9}
\mathcal{B}({u})=\int_{S}\sum_{|\alpha|\le 4}{|D^{\alpha}_{\xi,\eta}{u}|}^2d\xi d\eta.
\end{equation}
%We will show that there is another quadratic form $\mathcal{C}({u})$ which is spectrally
%equivalent to $\mathcal{B}({u})$ and can be diagonalized using separation of variables.

%Let $u(\xi,\eta)$ be a polynomial as in (\ref{eq5.8}).
 Define the
quadratic form
\begin{align}\label{eq5.19}
\mathcal{C}(u)=\int_S(&u_{\xi\xi\xi\xi}^2+u_{\eta\eta\eta\eta}^2
+u_{\xi\xi\xi}^2+u_{\eta\eta\eta}^2+u_{\xi\xi}^2
+u_{\eta\eta}^2+u_{\xi}^2+u_{\eta}^2+u^2)\:d\xi d\eta\;.
\end{align}
We now show that the quadratic form $\mathcal{C}(u)$ is spectrally equivalent to the
quadratic form $\mathcal{B}(u)$, defined in (\ref{eq5.7}) and can be diagonalized in
the basis $\psi_{i,j}(\xi,\eta)$. Note that ${\{\psi_{i,j}(\xi,\eta)\}}_{i,j}$
is the tensor product of the polynomials $\phi_i(\xi)$ and $\phi_j(\eta)$.% The eigenvalue
%$\mu_{i,j}$ corresponding to the eigenvector $\psi_{i,j}$ is given by the relation
%\begin{equation}\label{eq5.20}
%\mu_{i,j}=\mu_i+\mu_j+1\;.
%\end{equation}
%Hence, the matrix corresponding to the quadratic form $\mathcal{C}(u)$ is easy to invert.

To prove that $\mathcal{B}(u)$ and $\mathcal{C}(u)$ are spectrally equivalent we need to show that
there is an extension $U({\xi,\eta})$ of $u({\xi,\eta})$ such that
$U({\xi,\eta})\in H^{4}(\mathbb{R}^2)$ and satisfies the estimate
\begin{align*}
\int_{\mathbb{R}^2}&\left(U_{\xi\xi\xi\xi}^2+U_{\eta\eta\eta\eta}^2
+U_{\xi\xi}^2+U_{\eta\eta}^2+U^2\right)\:d\xi d\eta
 \leq K\int_{S} \left(u_{\xi\xi\xi\xi}^2+u_{\eta\eta\eta\eta}^2
+u_{\xi\xi}^2+u_{\eta\eta}^2+u^2\right)\:d\xi d\eta.
\end{align*}
Here, $K$ is a constant which is independent of $N$.

To extend $u({\xi,\eta})$ defined on $(-1,1)\times (-1,1)$ the method of successive reflections
is used. In the first step an extension $U_1({\xi,\eta})$ is obtained by reflecting $u({\xi,\eta})$
along the line $\eta=1$. This construction is similar to that in Theorem $5.19$ of~\cite{ADAM}.

For $\nu>0$ define
\[U_1(\xi,1+\nu)=\displaystyle\sum_{l=1}^{5} a_l u(\xi,1-l\nu)\Theta(\nu),\quad k=0,1,\ldots,4.\]
Here, $\Theta(\nu)$ is a $C^{\infty}$ function such that
\[\Theta(\nu)=1\quad \mbox{for}\quad \nu\leq \frac{1}{6}, \quad \Theta(\nu)=0\quad \mbox{for}\quad
\nu\geq \frac{1}{3}.\]
In addition, the coefficients $a_i$, $1\leq i \leq 5$ are chosen to satisfy the $5\times 5$ system of
linear equations
\[\displaystyle\sum_{l=1}^{5}(-l)^k a_l =1, \quad \quad k=0,1,\ldots,4.\]
%Now, we define the extension $U_1({\xi,\eta})$ of $u({\xi,\eta})$
Thus, the extension $U_1({\xi,\eta})$ of $u({\xi,\eta})$ can be written as
\begin{align}\label{eq2.15a}
U_1({\xi,\eta}) =\left\{\begin{array}{cc}
u({\xi,\eta}), \hspace{2.0cm} -1<\eta<1\\
\displaystyle\sum_{l=1}^{5} a_l u(\xi,1-l(\eta-1))\Theta(\eta-1),\quad
\eta\geq 1. \end{array} \right.
\end{align}
Therefore, using (\ref{eq2.15a}) we can write
\begin{align}
\int_{-1}^{\infty}&\int_{-1}^{1}\left((U_1)^2_{\xi\xi\xi\xi}+(U_1)^2_{\eta\eta\eta\eta}
+(U_1)^2_{\xi\xi}+(U_1)^2_{\eta\eta}+(U_1)^2\right)\:d\xi d\eta.\notag\\
&=\int_{S}\left(u^2_{\xi\xi\xi\xi}+u^2_{\eta\eta\eta\eta}+u^2_{\xi\xi}
+u^2_{\eta\eta}+u^2\right)\:d\xi d\eta\notag\\
&+\int_{1}^{\infty}\int_{-1}^{1}\left((U_1)^2_{\xi\xi\xi\xi}+(U_1)^2_{\eta\eta\eta\eta}
+(U_1)^2_{\xi\xi}+(U_1)^2_{\eta\eta}+(U_1)^2\right)\:d\xi d\eta\notag
\end{align}
Hence, we obtain %there exists a constant $K_1$ independent of $N$ such that
\begin{align}\label{eq2.15}
\int_{-1}^{\infty}&\int_{-1}^{1}\left((U_1)^2_{\xi\xi\xi\xi}+(U_1)^2_{\eta\eta\eta\eta}
+(U_1)^2_{\xi\xi}+(U_1)^2_{\eta\eta}+(U_1)^2\right)\:d\xi d\eta.\notag\\
&\leq K_1\int_{S}\left(u^2_{\xi\xi\xi\xi}+u^2_{\eta\eta\eta\eta}+u^2_{\eta\eta\eta}
+u^2_{\xi\xi}+u^2_{\eta\eta}+u^2_{\eta}+u^2\right)\:d\xi d\eta.
\end{align}
Here $K_1$ is a constant which is independent of $W$. \\
Appyling Lemma $5.6$ of~\cite{ADAM}, the following estimates hold
%there are constants $C_1, C_2$ such that
\begin{align}\label{eq2.16}
\int_{-1}^{1}u_{\eta}^2(\xi,\eta)\:d\eta \leq
C_1 \int_{-1}^{1}(u_{\eta\eta}^2+u^2)\:d\eta\;,
\end{align}
and
\begin{align}\label{eq2.17}
\int_{-1}^{1}u_{\eta\eta\eta}^2(\xi,\eta)\:d\eta \leq
C_2 \int_{-1}^{1}(u_{\eta\eta\eta\eta}^2+u^2)\:d\eta,
\end{align}
where $C_1$ and $C_2$ are constants.\\
Integrating (\ref{eq2.16}), (\ref{eq2.17}) with respect to $\xi$, we get
\begin{align}\label{eq2.18}
\int_{S}u_{\eta}^2(\xi,\eta)\:d\xi d\eta \leq
C_1 \int_{S}(u_{\eta\eta}^2+u^2)\:d\xi d\eta\;,
\end{align}
and
\begin{align}\label{eq2.19}
\int_{S}u_{\eta\eta\eta}^2(\xi,\eta)\:d\xi d\eta \leq
C_2 \int_{S}(u_{\eta\eta\eta\eta}^2+u^2)\:d\xi d\eta.
\end{align}
Combining (\ref{eq2.15}), (\ref{eq2.16}) and (\ref{eq2.17}), imply the following estimate
\begin{align}\label{eq2.20}
\int_{-1}^{\infty}&\int_{-1}^{1}\left((U_1)^2_{\xi\xi\xi\xi}
+(U_1)^2_{\eta\eta\eta\eta}+(U_1)^2_{\xi\xi}
+(U_1)^2_{\eta\eta}+(U_1)^2\right)\:d\xi d\eta
\leq C \int_{S}\left(u^2_{\xi\xi\xi\xi}+u^2_{\eta\eta
\eta\eta}+u^2_{\xi\xi}+u^2_{\eta\eta}+u^2\right)
\:d\xi d\eta.
\end{align}
Here, $C$ is a generic constant. We are now in a position to prove the lemma.
\begin{lem}\label{lem21}
Let $u(\xi,\eta)$ be the polynomial as defined in (\ref{eq5.8}). Then there is an
extension $(Eu)(\xi,\eta)=U(\xi,\eta)$ of $u(\xi,\eta)$ such that $U(\xi,\eta)\in
H^4(\mathbb{R}^2)$ and satisfies the following estimate
\begin{align}\label{eq2.21}
\mathop{\int}_{\mathbb{R}^2}&\left(U^2_{\xi\xi\xi\xi}
+U^2_{\eta\eta\eta\eta}+U^2_{\xi\xi}
+U^2_{\eta\eta}+U^2\right)\:d\xi d\eta
\leq K \mathop{\int}_{S}\left(u^2_{\xi\xi\xi\xi}
+u^2_{\eta\eta\eta\eta}+u^2_{\xi\xi}
+u^2_{\eta\eta}+u^2\right)\:d\xi d\eta,
\end{align}
\end{lem}
where $K$ is a constant independent of $W$.
\begin{proof}
Assume that $U_1(\xi,\eta)$ be the extension of $u(\xi,\eta)$, defined on
$(-1, 1)\times(-1,\infty)$, obtained by reflecting $u(\xi,\eta)$ about the line
$\eta = 1$. Then
\begin{align}\label{eq2.22}
\int_{-1}^{\infty}&\int_{-1}^{1}\left((U_1)^2_{\xi\xi\xi\xi}
+(U_1)^2_{\eta\eta\eta\eta}+(U_1)^2_{\xi\xi}
+(U_1)^2_{\eta\eta}+(U_1)^2\right)\:d\xi d\eta\notag\\
&\leq K_1\int_{S}\left(u^2_{\xi\xi\xi\xi}+u^2_{\eta\eta
\eta\eta}+u^2_{\xi\xi}+u^2_{\eta\eta}+u^2\right)
\:d\xi d\eta.
\end{align}
Let $U_2(\xi,\eta)$ to be the extension of $U_1(\xi,\eta)$,
defined on $(-1, 1)\times(-\infty,\infty)$, obtained by reflecting $U_1(\xi,\eta)$
about the line $\eta=-1$ then we have
\begin{align}\label{eq2.23}
\int_{-\infty}^{\infty}&\int_{-1}^{1}\left((U_2)^2_{\xi\xi\xi\xi}
+(U_2)^2_{\eta\eta\eta\eta}+(U_2)^2_{\xi\xi}
+(U_2)^2_{\eta\eta}+(U_2)^2\right)\:d\xi d\eta\notag\\
&\leq K_2\int_{-1}^{\infty}\int_{-1}^{1}\left((U_1)^2_{\xi\xi\xi\xi}
+(U_1)^2_{\eta\eta\eta\eta}+(U_1)^2_{\xi\xi}
+(U_1)^2_{\eta\eta}+(U_1)^2)\right)\:d\xi d\eta.
\end{align}
Let $U_3(\xi,\eta)$ to be the extension of $U_2(\xi,\eta)$, defined on
$(-1, \infty)\times(-\infty,\infty)$, obtained by reflecting $U_2(\xi,\eta)$
about the line $\xi=1$. Clearly,
\begin{align}\label{eq2.24}
\int_{-\infty}^{\infty}&\int_{-1}^{\infty}\left((U_3)^2_{\xi\xi\xi\xi}
+(U_3)^2_{\eta\eta\eta\eta}+(U_3)^2_{\xi\xi}
+(U_3)^2_{\eta\eta}+(U_3)^2\right)\:d\xi d\eta\notag\\
&\leq K_3\int_{-\infty}^{\infty}\int_{-1}^{1}\left((U_2)^2_{\xi\xi\xi\xi}
+(U_2)^2_{\eta\eta\eta\eta}+(U_2)^2_{\xi\xi}
+(U_2)^2_{\eta\eta}+(U_2)^2)\right)\:d\xi d\eta.
\end{align}
Finally, let $U(\xi,\eta)$ to be the extension of $U_3(\xi,\eta)$, defined on
$(-\infty, \infty)\times(-\infty,\infty)$, obtained by reflecting $U_3(\xi,\eta)$
about the line $\xi=-1$. It follows that the estimate
\begin{align}\label{eq2.25}
\int_{-\infty}^{\infty}&\int_{-\infty}^{\infty}\left(U^2_{\xi\xi\xi\xi}
+U^2_{\eta\eta\eta\eta}+U^2_{\xi\xi}
+U^2_{\eta\eta}+U^2\right)\:d\xi d\eta\notag\\
&\leq K_4\int_{-\infty}^{\infty}\int_{-1}^{\infty}\left((U_3)^2_{\xi\xi\xi\xi}
+(U_3)^2_{\eta\eta\eta\eta}+(U_3)^2_{\xi\xi}
+(U_3)^2_{\eta\eta}+(U_3)^2)\right)\:d\xi d\eta.
\end{align}
Combining (\ref{eq2.22}), (\ref{eq2.23}), (\ref{eq2.24}) and (\ref{eq2.25}) and choosing $K=K_1K_2K_3K_4$, imply the final estimate
(\ref{eq2.21}).  %we obtain with $K=K_1K_2K_3K_4$.
\end{proof}
\begin{theorem}\label{them11}
The quadratic forms $\mathcal{B}(u)$ and $\mathcal{C}(u)$ are spectrally equivalent.
\end{theorem}
\begin{proof}
Assume that $U(\xi,\eta)$ be the extension of $u(\xi,\eta)$ as defined in Lemma
\ref{lem21} with  $U(\xi,\eta)|_S=u(\xi,\eta)$.
% Then $u(\xi,\eta)$ is the restriction of $U(\xi,\eta)$ to $S$.
Furthermore, $U(\xi,\eta)\in H^4(\mathbb{R}^2)$ and satisfies the following estimate
\begin{align*}
\mathop{\int}_{\mathbb{R}^2}&\left(U^2_{\xi\xi\xi\xi}+U^2_{\eta\eta\eta\eta}
+U^2_{\xi\xi}+U^2_{\eta\eta}+U^2\right)\:d\xi d\eta\notag\\
&\leq K\mathop{\int}_{S}\left(u^2_{\xi\xi\xi\xi}+u^2_{\eta\eta\eta\eta}+u^2_{\xi\xi}
+u^2_{\eta\eta}+u^2\right)\:d\xi d\eta.
\end{align*}
Let $\widehat{U}(\alpha,\beta)$ be the Fourier transform of $U(\xi,\eta)$.
Then
\begin{align*}
\mathop{\int}_{\mathbb{R}^2} U^2_{\xi\xi\eta\eta}\:d\xi d\eta
&=\mathop{\int}_{\mathbb{R}^2}|\alpha^2\beta^2|^2|\widehat{U}(\alpha,\beta)|^2 \:d\alpha d\beta\leq \left(\mathop{\int}_{\mathbb{R}^2}|\alpha^4|^2|\widehat{U}(\alpha,\beta)|^2
\:d\alpha d\beta\right)^{\frac{1}{2}}\left(\mathop{\int}_{\mathbb{R}^2}|\beta^4|^2
|\widehat{U}(\alpha,\beta)|^2\:d\alpha d\beta\right)^{\frac{1}{2}}
\end{align*}
using Cauchy-Schwarz inequality.\\
Applying inverse Fourier transform, implies the following estimate
%Hence, the estimate is as follows:
\begin{align*}
\mathop{\int}_{\mathbb{R}^2} |U_{\xi\xi\eta\eta}|^2\:d\xi d\eta
&\leq \left(\mathop{\int}_{\mathbb{R}^2}|{U}_{\xi\xi\xi\xi}|^2
\:d\xi d\eta\right)^{\frac{1}{2}}\left(\mathop{\int}_{\mathbb{R}^2}
|{U}_{\eta\eta\eta\eta}|^2\:d\xi d\eta\right)^{\frac{1}{2}}.
\end{align*}
Using AM-GM inequality, we obtain
%Thus, we obtain
\begin{align}\label{newth1}
\mathop{\int}_{\mathbb{R}^2}|U_{\xi\xi\eta\eta}|^2\:d\xi d\eta
&\leq \frac{1}{2}\left(\mathop{\int}_{\mathbb{R}^2}|{U}_{\xi\xi\xi\xi}|^2
+|{U}_{\eta\eta\eta\eta}|^2\:d\xi d\eta\right).
\end{align}
Therefore, we have
\begin{align*}
\mathop{\int}_{S}|u_{\xi\xi\eta\eta}|^2\:d\xi d\eta
&\leq \frac{1}{2}\left(\mathop{\int}_{\mathbb{R}^2}|{U}_{\xi\xi\xi\xi}|^2
+|{U}_{\eta\eta\eta\eta}|^2\:d\xi d\eta\right).
\end{align*}
Inserting the result of Lemma $2.1$, the following estimate holds
\begin{align}\label{newth4}
\mathop{\int}_{S}|u_{\xi\eta\eta\eta}|^2\:d\xi d\eta
&\leq \frac{K}{2}\left(\mathop{\int}_{S}|u_{\xi\xi\xi\xi}|^2
+|{u}_{\eta\eta\eta\eta}|^2+|u_{\xi\xi}|^2
+|{u}_{\eta\eta}|^2+|u|^2\:d\xi d\eta\right).
\end{align}
%using Lemma $2.1$.\\
Now, we estimate the following term
\begin{align*}
\mathop{\int}_{\mathbb{R}^2} U^2_{\xi\eta\eta\eta}\:d\xi d\eta
&=\mathop{\int}_{\mathbb{R}^2}|\alpha\beta^3|^2|\widehat{U}(\alpha,\beta)|^2 \:d\alpha d\beta\leq \left(\mathop{\int}_{\mathbb{R}^2}|\alpha\beta|^4|\widehat{U}(\alpha,\beta)|^2
\:d\alpha d\beta\right)^{\frac{1}{2}}\left(\mathop{\int}_{\mathbb{R}^2}|\beta^4|^2
|\widehat{U}(\alpha,\beta)|^2\:d\alpha d\beta\right)^{\frac{1}{2}}.
\end{align*}
using Cauchy-Schwarz inequality.\\
Applying inverse fourier transform, implies the following estimate
\begin{align*}
\mathop{\int}_{\mathbb{R}^2} |U_{\xi\eta\eta\eta}|^2\:d\xi d\eta
&\leq \left(\mathop{\int}_{\mathbb{R}^2}|{U}_{\xi\xi\eta\eta}|^2
\:d\xi d\eta\right)^{\frac{1}{2}}\left(\mathop{\int}_{\mathbb{R}^2}
|{U}_{\eta\eta\eta\eta}|^2\:d\xi d\eta\right)^{\frac{1}{2}}.
\end{align*}
Using AM-GM inequality, we obtain
\begin{align*}
\mathop{\int}_{\mathbb{R}^2}|U_{\xi\eta\eta\eta}|^2\:d\xi d\eta
&\leq \frac{1}{2}\left(\mathop{\int}_{\mathbb{R}^2}|{U}_{\xi\xi\eta\eta}|^2
+|{U}_{\eta\eta\eta\eta}|^2\:d\xi d\eta\right).
\end{align*}
Moreover, we obtain
\begin{align}\label{newth2}
\mathop{\int}_{S}|u_{\xi\eta\eta\eta}|^2\:d\xi d\eta
&\leq \frac{1}{2}\left(\mathop{\int}_{\mathbb{R}^2}|{U}_{\xi\xi\eta\eta}|^2
+|{U}_{\eta\eta\eta\eta}|^2\:d\xi d\eta\right).
\end{align}
Applying (\ref{newth1}) and  Lemma $2.1$ in (\ref{newth2}), implies
\begin{align}\label{eq2.26}
\mathop{\int}_{S}|u_{\xi\eta\eta\eta}|^2\:d\xi d\eta
&\leq \frac{K}{2}\left(\mathop{\int}_{S}|u_{\xi\xi\xi\xi}|^2
+|{u}_{\eta\eta\eta\eta}|^2+|u|^2\:d\xi d\eta\right).
\end{align}
Similarly,
\begin{align*}
\mathop{\int}_{\mathbb{R}^2} U^2_{\xi\xi\xi\eta}\:d\xi d\eta
&=\mathop{\int}_{\mathbb{R}^2}|\alpha^3\beta|^2|\widehat{U}(\alpha,\beta)|^2 \:d\alpha d\beta\leq \left(\mathop{\int}_{\mathbb{R}^2}|\alpha^4|^2|\widehat{U}(\alpha,\beta)|^2
\:d\alpha d\beta\right)^{\frac{1}{2}}\left(\mathop{\int}_{\mathbb{R}^2}|\alpha\beta|^4
|\widehat{U}(\alpha,\beta)|^2\:d\alpha d\beta\right)^{\frac{1}{2}}.
\end{align*}
using Cauchy-Schwarz inequality.\\
Proceeding as above we obtain,
\begin{align}\label{newth3}
\mathop{\int}_{S}|u_{\xi\xi\xi\eta}|^2\:d\xi d\eta
&\leq \frac{1}{2}\left(\mathop{\int}_{\mathbb{R}^2}\left(|{U}_{\xi\xi\eta\eta}|^2
+|{U}_{\eta\eta\eta\eta}|^2\right)\:d\xi d\eta\right).
\end{align}
Again, Applying (\ref{newth1}) and  Lemma $2.1$ in (\ref{newth3}), implies
\begin{align}\label{eq2.27}
\mathop{\int}_{S}|u_{\xi\xi\xi\eta}|^2\:d\xi d\eta
&\leq \frac{K}{2}\mathop{\int}_{S}\left(|u_{\xi\xi\xi\xi}|^2
+|{u}_{\eta\eta\eta\eta}|^2+|u|^2\right)\:d\xi d\eta.
\end{align}
In a similar way it follows that,
%\begin{align}\label{eq2.28}
%\mathop{\int}_{S}|u_{\xi\xi\eta\eta}|^2\:d\xi d\eta
%&\leq \frac{K}{2}\mathop{\int}_{S}\left(|u_{\xi\xi\xi\xi}|^2
%+|{u}_{\xi\xi\eta\eta}|^2+|u|^2\right)\:d\xi d\eta\;,
%\end{align}
%\begin{align}\label{eq2.29}
%\mathop{\int}_{S}|u_{\xi\eta\xi\eta}|^2\:d\xi d\eta
%&\leq \frac{K}{2}\mathop{\int}_{S}\left(|u_{\xi\xi\xi\xi}|^2
%+|{u}_{\xi\xi\eta\eta}|^2+|u|^2\right)\:d\xi d\eta\;,
%\end{align}
\begin{align}\label{eq2.30}
\mathop{\int}_{S}|u_{\xi\eta\eta}|^2\:d\xi d\eta
&\leq \frac{K}{2}\mathop{\int}_{S}\left(|u_{\xi\xi}|^2
+|{u}_{\eta\eta\eta\eta}|^2+|u|^2\right)\:d\xi d\eta\;,
\end{align}
\begin{align}\label{eq2.31}
\mathop{\int}_{S}|u_{\xi\xi\eta}|^2\:d\xi d\eta
&\leq \frac{K}{2}\mathop{\int}_{S}\left(|u_{\xi\xi\xi\xi}|^2
+|{u}_{\eta\eta}|^2+|u|^2\right)\:d\xi d\eta\;,
\end{align}
and
\begin{align}\label{eq2.32}
\mathop{\int}_{S}|u_{\xi\eta}|^2\:d\xi d\eta
&\leq \frac{K}{2}\mathop{\int}_{S}\left(|u_{\xi\xi}|^2
+|{u}_{\eta\eta}|^2+|u|^2\right)\:d\xi d\eta.
\end{align}
Here, $K$ is a generic constant independent of $N$. Using (\ref{newth4}-\ref{eq2.32}),
we conclude that
\begin{align*}
\frac{1}{M}||u||_{H^4(S)}^2 & \leq \mathop{\int}_{S}(|u_{\xi\xi\xi\xi}|^2
+|u_{\eta\eta\eta\eta}|^2+|u_{\xi\xi\xi}|^2
+|u_{\eta\eta\eta}|^2+|u_{\xi\xi}|^2+|{u}_{\eta\eta}|^2\\
&+|u_{\xi}|^2+|{u}_{\eta}|^2+|u|^2)\:d\xi d\eta\leq ||u||_{H^4(S)}^2.
\end{align*}
The constant $M$ in the above is independent of $N$ and this proves the Lemma.
\end{proof}
We now show that the quadratic form $\mathcal{C}({u})$ defined in (\ref{eq5.19}) as
\begin{align}
\mathcal{C}(u)=\int_S(&u_{\xi\xi\xi\xi}^2+u_{\eta\eta\eta\eta}^2
+u_{\xi\xi\xi}^2+u_{\eta\eta\eta}^2+u_{\xi\xi}^2+u_{\eta\eta}^2
+u_{\xi}^2+u_{\eta}^2+u^2)\:d\xi d\eta\;.\notag
\end{align}
can be diagonalized in the basis $\{\psi_{i,j}\}_{i,j}$. Here, ${u}$ is a polynomial
in $\xi$ and $\eta$ as defined in (\ref{eq5.8}). Let $\widetilde{\mathcal{C}}(f,g)$
denote the bilinear form induced by the quadratic form $\mathcal{C}({u})$. Then
\begin{align}\label{eq5.21}
\widetilde{\mathcal{C}}(f,g)=\int_Q&\left(f_{\xi\xi\xi\xi}g_{\xi\xi\xi\xi}
+f_{\eta\eta\eta\eta}g_{\eta\eta\eta\eta}+f_{\xi\xi\xi}
g_{\xi\xi\xi}+f_{\eta\eta\eta}g_{\eta\eta\eta}\right.\notag\\
&\left.+f_{\xi\xi}g_{\xi\xi}
+f_{\eta\eta}g_{\eta\eta}+f_{\xi}g_{\xi}+f_{\eta}g_{\eta}+fg\right)
\:d\xi d\eta.
\end{align}
Let $I$ denote the interval $(-1,1)$ and
\begin{equation}\label{eq5.10}
v(\xi) = \sum_{i=0}^{W}\beta_{i}L_{i}(\xi) \;.
\end{equation}
Define $b={(\beta_{0},\beta_{1},\ldots,\beta_{W})}^T $. Next, the quadratic forms $\mathcal{E}(v)$ and $\mathcal{F}(v)$ are defined as
\begin{equation}\label{eq5.11}
\mathcal{E}(v) = \int_{I}(v_{\xi\xi\xi\xi}^2+v_{\xi\xi\xi}^2
+v_{\xi\xi}^2+v_{\xi}^2 )d\xi\;,
\end{equation}
and
\begin{equation}\label{eq5.12}
\mathcal{F}(v) = \int_{I} v^2 d\xi \;.
\end{equation}
Apparently, there exist $(W+1)\times(W+1)$ matrices $E$ and $F$ such that
\begin{equation}\label{eq5.13}
\mathcal{E}(v) = b^T E b\;,
\end{equation}
and
\begin{equation}\label{eq5.14}
\mathcal{F}(v) = b^T F b.
\end{equation}
Here, the matrices $E$ and $F$ are symmetric and $F$ is positive definite.
\newline
Furthermore, there exist $W+1$ eigenvalues $0 \le \mu_0 \le \mu_1\le\cdots\le \mu_W$ and $W+1$
eigenvectors $b_0, b_1,\ldots,b_W$ of the symmetric eigenvalue problem
\begin{equation}\label{eq5.15}
(E-\mu F)b=0  \;.
\end{equation}
Therefore, we have
$$(E-\mu_i F)b_i=0\;.$$
Here the eigenvectors $b_i$ are normalized. Hence the following relations hold
\begin{subequations}\label{eq5.16}
\begin{equation}\label{eq5.16a}
b_i^T F b_j= \delta^i_j  \;.
\end{equation}
and
\begin{align}\label{eq5.16b}
b_i^T E b_j=\mu_i \delta^i_j\;,
\end{align}
\end{subequations}
Define $b_i = (b_{i,0}, b_{i,1},\ldots, b_{i,W})$. Next, the polynomial $\phi_i(\xi)$ is defined as
\begin{equation}\label{eq5.17}
\phi_i(\xi)=\sum_{j=0}^{W}b_{i,j}L_j(\xi)\:\:\mbox{for}\:\:0\le i\le W\:.
\end{equation}
Now, we define the polynomial $\psi_{i,j}$ which is as follows:
%Next, let $\psi_{i,j}$ denote the polynomial
\begin{equation}\label{eq5.18}
\psi_{i,j}(\xi,\eta)=\phi_i(\xi)\phi_j(\eta)\;,
\end{equation}
for\quad $0 \leq i \le W,\;0 \le j\le W$.\\
%Let $\mathcal{E}(v)$ and $\mathcal{F}(v)$, respectively, be the quadratic forms
%defined in (\ref{eq5.11}) and (\ref{eq5.12}) and
Let $\widetilde{\mathcal{E}}(f,g)$
and $\widetilde{\mathcal{F}}(f,g)$ denote the corresponding bilinear forms induced
by $\mathcal{E}(v)$ and $\mathcal{F}(v)$. Then
\begin{subequations}\label{eq5.22}
\begin{align}\label{eq5.22a}
\widetilde{\mathcal{E}}(f,g)=\int_{I}(f_{\xi\xi\xi\xi}g_{\xi\xi\xi\xi}
+f_{\xi\xi\xi}g_{\xi\xi\xi}+f_{\xi\xi}g_{\xi\xi}
+f_{\xi}g_{\xi})\:d\xi\;,
\end{align}
and
\begin{align}\label{eq5.22b}
\widetilde{\mathcal{F}}(f,g)=\int_{I}fg\:d\xi\:,
\end{align}
\end{subequations}
%Here, $I$ denotes the unit interval and
 where $f(\xi)$ and $g(\xi)$ are polynomials of degree
$W$ in $\xi$.
\newline
Furthermore, the relation
(\ref{eq5.16a}) and (\ref{eq5.16b}) become
\begin{subequations}\label{eq5.23}
\begin{align}\label{eq5.23a}
\widetilde{\mathcal{F}}(\phi_i,\phi_j)=\int_{I}\phi_i(\xi)\phi_j(\xi)
\:d\xi=\delta_{j}^{i}\:.
\end{align}
and
%Moreover, relation (\ref{eq5.16b}) may be stated as
\begin{align}\label{eq5.23b}
\widetilde{\mathcal{E}}(\phi_i,\phi_j)&=\int_{I}((\phi_i)_{\xi\xi\xi\xi}
(\phi_j)_{\xi\xi\xi\xi}+(\phi_i)_{\xi\xi\xi}(\phi_j)_{\xi\xi\xi}
+(\phi_i)_{\xi\xi}(\phi_j)_{\xi\xi}+(\phi_i)_{\xi}(\phi_j)_{\xi})\:d\xi
=\mu_{i}\delta_{j}^{i}.
\end{align}
\end{subequations}
%Recall that $\psi_{i,j}(\xi,\eta)=\phi_{i}(\xi)\phi_{j}(\eta)$ and
Combining (\ref{eq5.18}),
(\ref{eq5.23}) and (\ref{eq5.21}), it is easy to show that
\begin{align}
\widetilde{\mathcal{C}}(\psi_{i,j},\psi_{k,l})&=(\mu_i+\mu_j+1)\delta_{k}^{i}
\delta_{l}^{j}=\mu_{i,j}\delta_{k}^{i}\delta_{l}^{j}\:. \notag
\end{align}
Clearly, the eigenvectors of the quadratic form $\mathcal{C}(u)$ are $\{\psi_{i,j}\}_{i,j}$
and the eigenvalues are $\{\mu_{i,j}\}_{i,j}$ given by the relation
\begin{equation}\label{eq5.20}
\mu_{i,j}=\mu_i+\mu_j+1\;.
\end{equation}
Moreover, the quadratic form
$\mathcal{C}({u})$ can be diagonalized in the basis ${\{\psi_{i,j}\}}_{i,j}$ and
consequently the matrix corresponding to $\mathcal{C}({u})$ is easy to invert.
% \begin{table}[!ht]
% \begin{center}
%          \begin{tabular}{|c|c|c|}
%             \hline
%              \text{W} & \text{$\log_{10}(\kappa)$}\\
%             \hline
%             \hline
%              4 & 2.0297311717987707 \\
%             \hline
%              8 & 2.1140187559336927 \\
%             \hline
%              12 & 2.1632580185237087 \\
%             \hline
%              16 & 2.1913619432963762 \\
%             \hline
%              20 & 2.2098240850855990 \\
%             \hline
%              24 & 2.2230698082644300 \\
%             \hline
%              28 & 2.2341619222714453 \\
%             \hline
%              32 & 2.2390002544849740 \\
%             \hline
%         \end{tabular}
% \end{center}
% \caption{Condition number $\kappa$ as a function of W}
% \label{tab5.1}
% \end{table}
%Hence the eigenvectors of the quadratic form $\mathcal{C}(u)$ are $\{\psi_{i,j}\}_{i,j}$
%and the eigenvalues are $\{\mu_{i,j}\}_{i,j}$. Moreover the quadratic form
%$\mathcal{C}({u})$ can be diagonalized in the basis ${\{\psi_{i,j}\}}_{i,j}$ and
%consequently the matrix corresponding to $\mathcal{C}({u})$ is easy to invert.

Let $\kappa$ denotes the condition number of the preconditioned system obtained by using the
quadratic form $\mathcal{C}(u)$ as a preconditioner for the quadratic form $\mathcal{B}(u)$.
Then the values of $\kappa$ as a function of $W$ are shown in Table \ref{tab5.1}.

% \begin{table}[!ht]
% \begin{center}
%          \begin{tabular}{|c|c|c|}
%             \hline
%              \text{W} & \text{$\log_{10}(\kappa)$}\\
%             \hline
%             \hline
%              4 & 2.0297311717987707 \\
%             \hline
%              8 & 2.1140187559336927 \\
%             \hline
%              12 & 2.1632580185237087 \\
%             \hline
%              16 & 2.1913619432963762 \\
%             \hline
%              20 & 2.2098240850855990 \\
%             \hline
%              24 & 2.2230698082644300 \\
%             \hline
%              28 & 2.2341619222714453 \\
%             \hline
%              32 & 2.2390002544849740 \\
%             \hline
%         \end{tabular}
% \end{center}
% \caption{Condition number $\kappa$ as a function of W}
% \label{tab5.1}
% \end{table}

\begin{table}[!ht]
\begin{center}
         \begin{tabular}{|c|c|}
            \hline
             \text{W} & \text{$\log_{10}(\kappa)$}\\
            \hline

             4 & 2.029731 \\
            \hline
             8 & 2.114018 \\
            \hline
             12 & 2.163258 \\
            \hline
             16 & 2.191361 \\
            \hline
             20 & 2.209824 \\
            \hline
             24 & 2.223069 \\
            \hline
             28 &  2.234161 \\
            \hline
             32 & 2.239001 \\
            \hline
%             36 &  \\
%            \hline
        \end{tabular}
\end{center}
\caption{Condition number $\kappa$ as a function of W}
\label{tab5.1}
\end{table}

\section{Solution Techniques}
Let
\[{u}(\xi,\eta)=\sum_{i=0}^{W}\sum_{j=0}^{W}\beta_{i,j}L_i(\xi)L_j(\eta)\;,\]
and $\beta$ denotes the column vector whose components are $\beta_{i,j}$ arranged in
lexicographic order. Then there is a ${(W+1)}^2\times{(W+1)}^2$ matrix $C$ such that
$$\mathcal{C}({u})=\beta^T C \beta\;.$$
We now show (as in~\cite{DBR}) how to solve the system of equations
\begin{equation}\label{eq112}
C\beta =\rho \notag.
\end{equation}
Define a polynomial $r(\xi,\eta)$ corresponding to the vector $\rho$ is given by
\[r(\xi,\eta)=\sum_{i=0}^{W}\sum_{j=0}^{W}\rho_{i,j}L_{i}(\xi)L_{j}(\eta).\]
Here, the column vector $\rho$ is obtained by arranging the elements $\rho_{i,j}$ in lexicographic order.\\
Now by (\ref{eq5.17})
\[\phi_{i}(\xi)=\sum_{j=0}^{W}b_{i,j}L_{j}(\xi)\quad\mbox{for}\quad0\leq i\leq W.\]
Inverting the above relation we have
\begin{equation}\label{eq111}
L_{i}(\xi)=\sum_{j=0}^{W}\tilde{b}_{i,j}\phi_{j}(\xi).
\end{equation}
Using (\ref{eq111}), we may write
\[r(\xi,\eta)=\sum_{i=0}^{W}\sum_{j=0}^{W}\tilde{\rho}_{i,j}\;\phi_{i}(\xi)\;\phi_{j}(\eta).\]
Next, we define the polynomial $g(\xi,\eta)$ as
\begin{align*}
g(\xi,\eta)&=\sum_{i=0}^{W}\sum_{j=0}^{W}\frac{\tilde{\rho}_{i,j}}{\mu_{i,j}}\;\phi_{i}(\xi)\;\phi_{j}(\eta)\;,\\
&=\sum_{i=0}^{W}\sum_{j=0}^{W}\nu_{i,j}\;\phi_{i}(\xi)\;\phi_{j}(\eta).
\end{align*}
Here, $\nu_{i,j} = \frac{\tilde{\rho}_{i,j}}{\mu_{i,j}}$. Now
\[g(\xi,\eta)=\sum_{i=0}^{W}\sum_{j=0}^{W}\beta_{i,j}\;L_{i}(\xi)\;L_{j}(\eta).\]
We can obtain $\{\beta_{i,j}\}_{i,j}$ from $\{\nu_{i,j}\}_{i,j}$ using relation (\ref{eq5.17}). Clearly, we
can solve the system of equations (\ref{eq112}) in $O(N^3)$ operations.

Let $u(\xi,\eta)$ be a polynomial as defined in (\ref{eq5.8}) which vanishes at the vertices of the square
$S = (-1, 1)\times(-1, 1)$.
\newline
Let
\[V_{1}(\xi)=\frac{(1-\xi)}{2}=\frac{L_{0}(\xi)-L_{1}(\xi)}{2},\]
\[V_{2}(\xi)=\frac{(1+\xi)}{2}=\frac{L_{0}(\xi)-L_{1}(\xi)}{2},\]
\begin{equation}\label{eq113}
V_{i}(\xi)=\sqrt{\frac{2i-3}{2}}\int_{-1}^{\xi}L_{i-2}(s)ds=\frac{1}{2\sqrt{(2i-3)}}(L_{i-1}(\xi)-L_{i-3}(\xi))
\quad \mbox{for}\quad\leq i\leq W+1,
\end{equation}
denote the hierarchic shape functions as defined in~\cite{SCHW}. Then
\[V(\pm 1)=0\quad \mbox{for}\quad 3\leq i\leq W+1\;,\]
and $V_{1}(\xi)$ vanishes at $\xi = 1$. Moreover, $V_{2}(\xi)$ vanishes at $\xi=-1$. Let
\[\omega(\xi)=\sum_{i=3}^{W+1}\gamma_{i}V_{i}(\xi)\;,\]
and $\mathcal{E}(\omega)$ and $\mathcal{F}(w)$ be the quadratic forms defined in (\ref{eq5.11}) and (\ref{eq5.12}).
Clearly, there exist $W - 1\times W - 1$ matrices $\tilde{{E}}$ and $\tilde{{F}}$ such that
\[\mathcal{E}=c^{T}\tilde{E}c\;,\]
and
\[\mathcal{F}=c^{T}\tilde{F}c.\]
Here, $c$ denotes the vector
\[c=[\gamma_{3},\gamma_{4},\ldots,\gamma_{W+1}]^{T}.\]
Moreover, the matrices $\tilde{E}$ and $\tilde{F}$ are symmetric and $\tilde{F}$ is positive definite.
Hence, there exist $W - 1$ eigenvalues
\[0\leq\nu_{3}\leq \nu_{4}\leq\cdots\leq\nu_{W+1}\;,\]
of the symmetric eigenvalue problem
\begin{equation}\label{eq114}
(\tilde{E}-\nu\tilde{F})c=0.
\end{equation}
Let $c_{i}$ be the eigenvector corresponding to the eigenvalue $\nu_{i}$ . Then
\begin{equation}\label{eq115}
(\tilde{E}-\nu_{i}\tilde{F})c_{i}=0.
\end{equation}
Moreover the eigenvectors $c_{i}$ are normalized so that
\begin{equation}\label{eq116}
c_{i}^{T}\tilde{F}c_{j}=\delta^{i}_{j}.
\end{equation}
In addition the relations
\begin{equation}\label{eq117}
c_{i}^{T}\tilde{E}c_{j}=\nu_{i}\delta^{i}_{j}\;,
\end{equation}
hold. Here $c_{i}=[c_{i,3},c_{4},\ldots,c_{i,W+1}]^{T}.$\\
We now define the polynomials
\begin{equation}\label{eq118}
h_{i}(\xi)=\sum_{j=3}^{W+1}c_{i,j}V_{j}(\xi)\quad \mbox{for}\quad 3\leq i\leq W+1.
\end{equation}
Then
\begin{subequations}\label{eq119}
\begin{align}\label{eq119a}
\int_{-1}^{1}h_{i}h_{j} d\xi = \delta^{i}_{j}\;,
\end{align}
and
\begin{align}\label{eq120b}
\int_{-1}^{1}((h_{i})_{\xi\xi\xi\xi}(h_{j})_{\xi\xi\xi\xi}+(h_{i})_{\xi\xi\xi}(h_{j})_{\xi\xi\xi}
+(h_{i})_{\xi\xi}(h_{j})_{\xi\xi}+(h_{i})_{\xi}(h_{j})_{\xi})d\xi=\nu_{i}\delta^{i}_{j}.
\end{align}
\end{subequations}
Now consider the bilinear form
\[\widetilde{\mathcal{C}}(f,g)=\int_{S}(f_{\xi\xi\xi\xi}g_{\xi\xi\xi\xi}+f_{\eta\eta\eta\eta}g_{\eta\eta\eta\eta}
+f_{\xi\xi\xi}g_{\xi\xi\xi}+f_{\eta\eta\eta}g_{\eta\eta\eta}+f_{\xi\xi}g_{\xi\xi}+f_{\eta\eta}g_{\eta\eta}
+f_{\xi}g_{\xi}+f_{\eta}g_{\eta})d\xi d\eta\;,\]
defined in (\ref{eq5.21}).
\newline
Let $P_{i,j}$ denote the polynomial
\begin{equation}\label{eq121}
P_{i,j}(\xi,\eta)=h_{i}(\xi)h_{j}(\eta)\;,
\end{equation}
for $3\leq i,j\leq W+1$.
Then, using relations (\ref{eq119}) it is easy to show that
\begin{equation}\label{eq122}
\widetilde{\mathcal{C}}(P_{i,j},P_{k,l})=(\nu_{i}+\nu_{j}+1)\delta^{i}_{j}\delta^{k}_{l}.
\end{equation}
Now, if $u(\xi,\eta)$ is a polynomial as defined in (\ref{eq5.8}) which vanishes at the vertices of
the square $S$ then it has the representation
\begin{align}\label{eq123}
u(\xi,\eta)&=\sum_{i=3}^{W+1}\sum_{j=3}^{W+1}o_{i,j} h_{i}(\xi)h_{j}(\eta) +\sum_{i=3}^{W+1}e_{i} h_{i}(\xi)V_{1}(\eta)+\sum_{i=3}^{W+1}f_{i} h_{i}(\xi)V_{2}(\eta)\notag\\
&+\sum_{j=3}^{W+1}g_{j} V_{1}(\xi)h_{j}(\eta)+\sum_{j=3}^{W+1}h_{j} V_{2}(\xi)h_{j}(\eta).
\end{align}
Or, we may write
\begin{equation}\label{eq124}
\begin{split}
u(\xi,\eta)=\sum_{i=3}^{W+1}\sum_{j=3}^{W+1}o_{i,j} P_{i,j}(\xi,\eta)+\sum_{i=1}^{4W-4}q_{i} R_{i}(\xi,\eta).
\end{split}
\end{equation}
Here $\{R(\xi,\eta)\}_{i=1,4W-4}$ denote the polynomials $\{h_{i}(\xi)V_{1}(\eta)\}_{i=3,\ldots,W+1}, \{h_{i}(\xi)V_{2}(\eta)\}_{i=3,\ldots,W+1}, \\ \{V_{1}(\xi)h_{j}(\eta)\}_{j=3,\ldots,W+1}$ and $\{V_{1}(\xi)h_{j}(\eta)\}_{j=3,\ldots,W+1}$.\\
Let $A$ denote the matrix of the bilinear form $\widetilde{\mathcal{C}}(f, g)$ in the basis consisting
of $\{P_{i,j} (\xi, \eta)\}_{i,j}, \{R_{i} (\xi, \eta)\}_{i}$. Then
\[ A=\left[ \begin{array}{cc}
D & E  \\
E^{T} & F  \\
\end{array} \right].\]
Here $D$ is a $(W - 1)^2\times(W - 1)^2$ matrix and $E$ is a $(W - 1)^2\times(4W - 4)$ matrix. Moreover
$F$ is a $(4W - 4)\times(4W - 4)$ matrix.
\newline
Let $o$ denote the vector whose components are $o_{i,j}$ arranged in lexicographic order and $q$ the vector
whose components are $q_{i}$ . Here $\{o_{i,j}\}_{i,j}$ and $\{q_{i} \}_{i}$ are as in (\ref{eq124}). Let
$p$ denote the vector
\[ p=\left[ \begin{array}{c}
o  \\
q  \\
\end{array} \right],\]
and $z$ denote the vector
\[ z=\left[ \begin{array}{c}
x  \\
y  \\
\end{array} \right].\]
We now wish to solve the system of equations
\begin{equation}\label{eq125}
Ap=z.
\end{equation}
Define the Schur complement $S$ of the system of equations (\ref{eq125}) as
\[S = F - E^{T}D^{-1}E.\]
Then $S$ is a $(4W - 4)\times(4W - 4)$ matrix.\\
To obtain the solution $p$ of (\ref{eq125}) we first solve
\begin{equation}\label{eq126}
Sq = y - E^{T}D^{-1}x.
\end{equation}
Next, we compute
\begin{equation}\label{eq127}
o=D^{-1}(x-Eq).
\end{equation}
It is easy to see that the system of equations (\ref{eq125}) can be solved in $O(W^{3})$ operations.
%Moreover the bounds on the condition number given in Table $1$ will continue to hold.

\section{Conclusions}
Preconditioners for fourth order elliptic using least-squares spectral element methods have been
presented  in this paper. Numerical results presented in Table $1$ demonstrate the effectiveness
of these preconditioners.
We remark that the results and the method presented in this paper can be easily extended with
obvious modifications to three-dimensional models and more general differential operators (e.g.,
variable coefficients, higher order, multidimensional differential operator).

\end{document}